\documentclass[reqno, qed]{amsart}

\usepackage{amssymb, amsmath}
\usepackage{mathrsfs}
\usepackage{amscd}
\usepackage{verbatim}
\usepackage{enumerate}
\usepackage{color}
\usepackage[normalem]{ulem}
\usepackage{cancel}

\usepackage{hyperref}
\allowdisplaybreaks

\theoremstyle{plain}
\newtheorem{thrm}{Theorem}[section]
\theoremstyle{remark}

\newtheorem{remark}[thrm]{Remark}
\newtheorem{example}[thrm]{Example}
\theoremstyle{plain}
\newtheorem{theorem}[thrm]{Theorem}

\newtheorem{corollary}[thrm]{Corollary}
\newtheorem{lemma}[thrm]{Lemma}
\newtheorem{proposition}[thrm]{Proposition}

\newtheorem{definition}[thrm]{Definition}

\numberwithin{equation}{section}

\newcommand{\leaderfill}{\leaders\hbox to 1em{\hss.\hss}\hfill}
\newcommand{\RR}{\mathbb{R}}
\newcommand{\R}{\mathbb{R}}

\renewcommand{\P}{{\mathbb P}}
\newcommand{\EE}{\mathbb{E}}
\newcommand{\E}{\mathbb{E}}

\newcommand{\e}{\varepsilon}
\newcommand{\one}{\mathbf{1}}

\newcommand{\calL}{{\mathscr L}}
\newcommand{\OO}{\Omega}

\newcommand{\F}{{\mathscr F}}

\newcommand{\g}{\gamma}
\newcommand{\lb}{\langle}
\newcommand{\rb}{\rangle}
\newcommand{\limn}{\lim_{n\to\infty}}

\newcommand{\umd}{\textsc{umd}}

\begin{document}

\title[Forward integration]{Forward integration, convergence and nonadapted pointwise multipliers}

\author{Matthijs Pronk and Mark Veraar}
\address{Delft Institute of Applied Mathematics\\
Delft University of Technology \\ 2600 GA Delft\\The
Netherlands} \email{Matthijs.Pronk@inphykem.com}
\address{Delft Institute of Applied Mathematics\\
Delft University of Technology \\ 2600 GA Delft\\The
Netherlands} \email{M.C.Veraar@tudelft.nl}

\begin{abstract}
In this paper we study the forward integral of operator-valued processes with respect to a cylindrical Brownian motion. In particular, we provide conditions under which the approximating sequence of processes of the forward integral, converges to the stochastic integral process with respect to Sobolev norms of smoothness $\alpha<1/2$. This result will be used to derive a new integration by parts formula for the forward integral.
\end{abstract}

%60G05  View Publications (1973-now) Foundations of stochastic processes
%60G07  View Publications (1980-now) General theory of processes
%60H05  View Publications (1973-now) Stochastic integrals
%60B12  View Publications (1980-now) Limit theorems for vector-valued random variables (infinite-dimensional case)
%46B09  View Publications (1991-now) Probabilistic methods in Banach space theory [See also 60Bxx]
%46E35  View Publications (1973-now) Sobolev spaces and other spaces of "smooth'' functions, embedding theorems, trace theorems
%60G17  View Publications (1973-now) Sample path properties

\thanks{The first named author is supported by VICI subsidy 639.033.604
of the Netherlands Organisation for Scientific Research (NWO)}



\maketitle

\section{Introduction}

In \cite{RV91} and \cite{RV93} Russo and Vallois initiated a theory of stochastic integration via regularization procedures. In later years this was further developed by them and several other authors (see \cite{KrukRuTu, RuTrut, NeNo, Tudor, EwXi, CGR, Jing}, and also the lecture notes \cite{RV-survey} and its references). The regularization procedure is connected to the celebrated forward and backward integrals which can be used to integrate with respect to more general processes than semimartingales. Applications arise for instance in the situation where the integrator is a fractional Brownian motion. Another feature is that the forward and backward integrals allow to integrate non-adapted processes.

Since the development of the Skorohod integral in \cite{Skorohod-integral}, integration of non-adapted integrands is used in the theory of SDEs (see \cite{Nualart2,NOP, Pardoux88, RV-survey} and references therein). A basic example where non-adapted integrands naturally occur is when the initial value of an SDE depends on the full paths of the underlying stochastic process (see \cite{Buckdahn91, MilletNualartSanz}). In many situations the forward integral is easier to work with than the Skorohod integral as a difficult correction term can often be avoided (see the It\^o formula in \cite{RV95}, \cite[Theorem 8.12]{NOP}).
The forward integral is used widely in the modeling of insider trading, which was introduced in \cite{BiagOks}. Since then, this has been further developed (see \cite[Chapter 8]{NOP} and its references). In particular, in \cite{NMBOP2, NMBOP1, OksZha} the authors generalized the forward integral to the setting of L\'{e}vy processes.

In the infinite dimensional setting several authors have worked on stochastic calculus for the Skorohod integral  (see \cite{janjan, MayZak1, NuPar, PV13} and references therein). However, only few results are available for the forward integral in infinite dimensions. In \cite{GirRusso}, Di Girolami and Russo  present a general set-up for an It\^o formula and covariation formulas. In \cite{NualartLeon} Le\'{o}n and Nualart have introduced the forward integral in the operator-valued setting and used it to study stochastic evolution equations in Hilbert spaces with an adapted (unbounded) drift.

In this paper we study several properties of the forward integral where the integrand is an operator-valued process and the integrator a cylindrical Wiener process.
We will prove a new approximation result for the forward integral (see Theorem \ref{thrm:frwrd-ito-coincide-p>2} and Corollary \ref{cor:type2} below). In the one-dimensional setting this result takes the following form:
\begin{theorem}\label{thm:intro}
Let $w$ be a standard Brownian motion and let $g$ be an adapted and measurable process with almost all paths in $L^p(0,T)$ with $p\geq 2$. Then
the pathwise defined process
\[t\mapsto n\int_0^t g(s) (w(s+\tfrac1n) - w(s)) \, ds, \ \ \ t\in [0,T]\]
converges to the It\^o integral process $t\mapsto \int_0^\cdot g\, d w$ in $W^{\alpha,p}(0,T)$ in probability for every $\alpha\in [0,\tfrac12)$.
\end{theorem}

The above result will be a particular case of two more general results on forward integration in {\umd} Banach spaces.
The class of {\umd} Banach spaces was extensively studied in the work of Burkholder (see \cite{Bu3} and references therein).
The {\umd} property plays an important role in both vector-valued stochastic and harmonic analysis. Stochastic integration and calculus in Banach spaces is naturally limited to the class of {\umd} Banach space (see \cite{BNVW08,NVW1}). Applications to stochastic evolution equations have been given in \cite{NVW-evolution} and several works afterwards (see the recent survey \cite{NVWsurvey} for further references).

As an application of the operator valued version of the above convergence result, we derive a new pointwise multiplier result for the forward integral (see Section \ref{sec:nonadaptedmult}).
It can be interpreted as an integration by parts formula.
The main novelty is that we can multiply adapted It\^o integrable processes with a process $M$ which is smooth in time but not necessarily adapted. Moreover, it is allowed to have a non-integrable singularity at $t=T$.
This result will be obtained in the operator-valued setting. It is particularly interesting in the study of mild solutions of non-autonomous stochastic evolution equations with adapted drift, where indeed the multiplier has a non-integrable singularity. A well-known obstacle in non-autonomous stochastic evolution equations with adapted drift is that the stochastic convolution term is not well-defined as an It\^o integral due to adaptedness problems.
In \cite{NualartLeon} this problem has been investigated using integration by parts for the Skorohod integral.
This formula for the Skorohod integrals can be obtained in the case $M$ is constant in time and satisfies certain Malliavin differentiability.
In \cite{PV-representation} we use the integration by parts formula of Theorem \ref{thm:productforward} to give a new approach to non-autonomous stochastic evolution equations with adapted drift.

\section{Preliminaries\label{sec:prel}}

In this paper we let $H$ be a separable Hilbert space and we fix an orthonormal basis $(h_n)_{n\geq 1}$. Let $T\in (0,\infty)$ be a fixed time and $X$ a {\umd} Banach space. All vector spaces will be assumed to be defined over the real scalar field, but with minor adjustments one can also allow complex scalars. We refer to \cite{Bu3} for details on {\umd} Banach spaces. The space $(\OO,\F, \P)$ will be a probability space with filtration $(\F_t)_{t\geq 0}$ and expectation is denoted by $\E$. Moreover, we write $L^0(\OO;X)$ for the strongly measurable functions $\xi:\OO\to X$ with the topology given by convergence in probability.
In the sequel $C$ will be a constant which may vary from line to line.

\subsection{Radonifying operators}
Let $\mathscr{H}$ be a real separable Hilbert space (below we take $\mathscr{H} = L^2(S;H)$ where $H$ is another Hilbert space).
We refer to \cite[Chapter 12]{DJT} and the survey paper \cite{Neerven-Radon} for an overview on $\gamma$-radonifying operators and unexplained terminology below. The Banach space of $\gamma$-radonifying operators from $\mathscr{H}$ into $X$ will be denoted by $\g(\mathscr{H},X)$. It is a subspace of $\calL(\mathscr{H},X)$. It satisfies the left- and right-ideal property. In particular, for $R\in \g(\mathscr{H},X)$, $U\in \calL(X)$ and $T\in \calL(\mathscr{H})$, one has $U R T\in \g(\mathscr{H},X)$ and
\[\|URT\|_{\g(\mathscr{H},X)}\leq \|U\| \, \|R\|_{\g(\mathscr{H},X)} \, \|T\|.\]
A simple consequence of the right-ideal property is that every operator $T \in \calL(\mathscr{H})$ has an extension to an operator
\begin{equation}\label{eq:gamma-ext}
\begin{split}
\tilde{T}:\g(\mathscr{H},X) &\to \g(\mathscr{H},X), \\
R&\mapsto R T^*,
\end{split}
\end{equation}
and $\|\tilde{T}\| = \|T\|$.

Let $(S,\Sigma,\mu)$ be a $\sigma$-finite measure space and $H$ be a Hilbert space. A function $G:S\to \calL(H,X)$ will be called {\em $H$-strongly measurable} if for all $h\in H$, $s\mapsto G(s) h$ is strongly measurable. Moreover, for $p\in (1, \infty)$, $G$ will be called {\em weakly $L^p(S;H)$} if for all $x^*\in X^*$, $s\mapsto G(s)^* x^*$ is in $L^p(S;H)$. For $G:S\to \calL(H,X)$ which is $H$-strongly measurable and weakly $L^2(S;H)$ we define $R_G:L^2(S;H)\to X$ as the (Pettis) integral operator
\begin{equation}\label{eq:weakint}
\lb R_G f, x^*\rb  = \int_S \lb G(s) f(s), x^*\rb \, d\mu(s),  \ \ \ f\in L^2(S;H), \ \ x^*\in X^*.
\end{equation}
Note that
\begin{equation}\label{eq:gamma-estm}
\|R_Gf\|_X \leq \|R_G\|_{\g(L^2(S;H),X)} \|f\|_{L^2(S;H)}.
\end{equation}
We will say $G\in \g(S;H,X)$ if $R_G\in \g(L^2(S;H),X)$ and write $\|G\|_{\g(S;H,X)} = \|R_G\|_{\g(L^2(S;H),X)}$.
It is well-known that the step functions $G:S\to \calL(H,X)$ of finite rank are dense in $\g(S;H,X)$. We will write $\g(0,T;X)$ for $\g((0,T);\RR,X)$.

For many operators $T:L^2(S;H)\to L^2(S;H)$ one has the property that $\tilde{T} R_G = R_F$ for a certain function $F$. In this case it will be convenient to write $T G = F$.

An easy consequence of the definitions and the ideal property is that \[\|G\one_{S_0}\|_{\g(S;H,X)} = \|G|_{S_0}\|_{\g(S_0;H,X)}.\] We will also use the following property.
\begin{example}
For $G\in \g(S;H,X)$ and $b\in L^\infty(S)$ one has $b G \in \g(S;H,X)$ and
\begin{equation}\label{eq:Linftyproduct}
\|b G\|_{\g(S;H,X)}\leq \|b\|_{L^\infty(S)} \|G\|_{\g(S;H,X)}.
\end{equation}
This is immediate from the right-ideal property with operator $T_b:L^2(S;H)\to L^2(S;H)$ given by $T_b f= b f$.
\end{example}

Finally we recall that in the special case that $X$ is a Hilbert space, one has
\begin{equation}\label{eq:HSid}
\gamma(S;H,X) = L^2(S;\mathcal{C}^2(H,X)),
\end{equation}
where $\mathcal{C}^2(H,X)$ denotes the space of Hilbert-Schmidt operators.

\begin{lemma}[$\gamma$-Integration by parts]\label{lem:intbyparts}
Let $M\in W^{1,1}(0,T;\calL(X))$.
Then for every $f\in \g(0,T;X)$ one has $M f\in \g(0,T;X)$ and for all $0\leq a<b\leq T$,
\begin{align}\label{eq:identityintbyparts}
\int_a^b M(s) f(s) \, ds = M(a) F(a) + \int_a^b M'(s) F(s) \, ds,
\end{align}
where $F(t) = \int_t^b f(s) \, ds$.
\end{lemma}
\begin{proof}
By \cite[Example]{KunstmanWeis} the family $\{M(t):t\in [0,T]\}$  is $R$-bounded by $C$. Therefore, by the Kalton--Weis $\gamma$-multiplier theorem (see \cite[Theorem 5.2]{Neerven-Radon}), one has that $M f\in \gamma(0,T;X)$ again and $\|M f\|_{\g(0,T;X)}\leq C \|f\|_{\g(0,T;X)}$. One also has
\[\|F(t)\|\leq \|f\|_{\g(0,T;X)} \|\one_{(t,b)}\|_{L^2(0,T)} \leq T^{1/2} \|f\|_{\g(0,T;X)}\]
and hence
\begin{align*}
\int_0^T \|M'(t) F(t)\| \, dt & \leq \|M\|_{W^{1,1}(0,T;\calL(X))} \sup_{t\in [0,T]}\|F(t)\|
\\ & \leq \|M\|_{W^{1,1}(0,T;\calL(X))} T^{1/2} \|f\|_{\g(0,T;X)}.
\end{align*}

For step functions $f:(0,T)\to X$, the identity \eqref{eq:identityintbyparts} is easy to verify. Now the general case follows from the above estimates and a density argument.
\end{proof}

\subsection{Integration with respect to a cylindrical Brownian motion}
Let $\mathscr{H} = L^2(0,T;H)$, where $H$ is a separable real Hilbert space. For details on stochastic integration in {\umd} Banach space we refer to \cite{NVW1,NVWsurvey}. The operator $W:\mathscr{H}\to L^2(\OO)$ will be called a {\em cylindrical Brownian motion} if for all choices $h\in \mathscr{H}$, $W h$ is a centered Gaussian random variable and for $h, \tilde{h}\in {\mathscr{H}}$, $\E(W h W \tilde{h}) = [h,\tilde{h}]$, where $[\cdot,\cdot]$ denotes the inner product on $\mathscr{H}$.

A process $G:(0,T)\times\OO\to \calL(H,X)$ will be called $H$-strongly adapted if for all $t\in (0,T)$ and $h\in H$, $\omega\mapsto G(t,\omega) h$ is strongly $\F_t$-measurable. If $G$ is $H$-strongly measurable and adapted, then from the separability of $H$ and \cite[Theorem 0.1]{OndrSe}, one can derive that $G$ has a version which is {\em $H$-strongly progressively measurable}, i.e. for each $h\in H$, $(t,\omega)\mapsto G(t,\omega) h$ is strongly progressively measurable. This will be used below without further notice.

Recall from \cite{CoxVer2,NVW1,NVWsurvey} that if $X$ is a {\umd} space and $G$ is an adapted process in $L^0(\OO;\gamma(0,T;H,X))$, one can define a stochastic integral
$I(G) = \int_0^T G\, d W$ in a natural way. We also let $J(G)(t) = \int_0^t G \, d W$, and recall that $J(G)$ has a version with continuous paths.
Moreover, for all $p\in (0,\infty)$ the following two-sided estimate for the stochastic integral holds:
\begin{align}\label{eq:Itoisom}
C^{-1}\|G\|_{L^p(\OO;\gamma(0,T;H,X))}\leq  \|J(G)\|_{L^p(\OO;C([0,T];X))} \leq C \|G\|_{L^p(\OO;\gamma(0,T;H,X))}.
\end{align}

\begin{remark}
All results below hold under the slightly weaker assumption that only the right-hand side of \eqref{eq:Itoisom} holds. This includes spaces such as $X = L^1$. For details on such spaces we refer to \cite{CoxVer2,CoxGei}.
\end{remark}

\subsection{Function spaces}

For $\alpha\in (0,1)$, $p\in [1, \infty)$ and $a<b$, recall that a function $f:(a,b)\to X$ is said to be in the {\em Sobolev space} $W^{\alpha,p}(a,b;X)$  if $f\in L^p(a,b;X)$ and
\[[f]_{W^{\alpha,p}(a,b;X)} := \Big(\int_a^b \int_a^b \frac{\|f(t) - f(s)\|^p}{|t-s|^{\alpha p + 1}} \, ds \, dt \Big)^{1/p} <\infty.\]
Letting $ \|f\|_{W^{\alpha,p}(a,b;X)} = \|f\|_{L^p(a,b;X)} + [f]_{W^{\alpha,p}(a,b;X)}$, this space becomes a Banach space. A function $f:(a,b)\to X$ is said to be in the {\em H\"older space} $C^{\alpha}(a,b;X)$ if
\[[f]_{C^{\alpha}(a,b;X)} = \sup_{a<s<t<b} \frac{\|f(t) - f(s)\|}{|t-s|^\alpha}<\infty.\]
Letting $\displaystyle \|f\|_{C^{\alpha}(a,b;X)} = \sup_{t\in (0,T)}\|f(t)\|_X + [f]_{W^{\alpha,p}(a,b;X)}$, this space becomes a Banach space. Moreover, every $f\in C^{\alpha}(a,b;X)$ has a unique extension to a continuous function $f:[a,b]\to X$.

If $0<\alpha<\beta<1$, then one easily checks
\begin{equation}\label{eq:fractionalSobolevembtrivial}
C^{\alpha}(a,b;X) \hookrightarrow W^{\alpha,p}(a,b;X).
\end{equation}
One of the main results in the theory of fractional Sobolev spaces is the following Sobolev embedding: if $\alpha>\frac1p$, then
\begin{equation}\label{eq:fractionalSobolevemb}
W^{\alpha,p}(a,b;X) \hookrightarrow C^{\alpha-\frac{1}{p}}(a,b;X).
\end{equation}
Here the embedding means that each $f\in W^{\alpha,p}(a,b;X)$ has a version which is continuous
and this function lies in $C^{\alpha-\frac{1}{p}}(a,b;X)$. The embedding \eqref{eq:fractionalSobolevemb} can be found in the
literature in the scalar setting and the standard proofs extend to the vector-valued setting. We refer to \cite[14.28 and 14.40]{Leoni09} and \cite[Theorem 8.2]{fracsob} for detailed proofs.

\section{Forward integral\label{sec:Forward}}
Recall that $H$ is a separable real Hilbert space with orthonormal basis $(h_n)_{n\geq 1}$. Let $P_n$ be the projection onto the first $n$ basis coordinates.

\begin{definition}\label{defn:fwrd-int}
Let $G: [0,T] \times \OO \to \calL(H,X)$ be $H$-strongly measurable and weakly in $L^2(0,T;H)$.
Define the sequence $(I^-(G,n))_{n=1}^\infty$ by
\begin{align*}
I^-(G,n) = \sum_{k=1}^n n\int_0^T G(s) h_k (W(s+1/n)h_k - W(s)h_k)\;ds,
\end{align*}
where the integral is defined as in \eqref{eq:weakint}.

The process $G$ is called {\em forward integrable} if $(I^-(G,n))_{n\geq 1}$ converges in probability.
In that case, the limit is called {\em the forward integral} of $G$ and is denoted by
\[I^{-}(G) = \int_0^T G\;d^-W = \int_0^T G(s)\;d^-W(s).\]
\end{definition}
Note that the above definition does not require any adaptedness properties of $G$. Unfortunately, it is unclear whether $I^{-}$ is a closable operator. For the Skorohod integral this is indeed the case (see \cite[Section 1.3]{Nualart2}).

We write $J^{-}(G,n)$ for the process given by
\begin{align}\label{eq:J-def}
J^{-}(G,n)(t) = I^{-}(G\one_{[0,t]},n).
\end{align}
Then $J^-(G,n) \in L^0(\OO;C^{1/2}(0,T;X))$. Indeed, by \eqref{eq:gamma-estm} we have a.s. for all $s<t$,
\begin{align*}
\|J^-(G,n)&(t) - J^-(G,n)(s)\| \leq \sum_{k=1}^n n \Big\| \int_s^t G(r)h_k (W(r+1/n)h_k - W(r)h_k)\;dr \Big\| \\
&\leq \sum_{k=1}^n n \|Gh_k\|_{\g(0,T;X)} \|r\mapsto \one_{[s,t]}(r)(W(r+1/n)h_k - W(r)h_k)\|_{L^2(0,T)}\\
&\leq 2(t-s)^{1/2} \sum_{k=1}^n n \|Gh_k\|_{\g(0,T;X)} \sup_{r\in [0,T+1/n]} |W(r)h_k|,
\end{align*}
and hence the result follows.

If for every $t\in [0,T]$, $(J^{-}(G,n)(t))_{n\geq 1}$ converges in probability, we write $J^{-}(G)$ for the process given by $J^{-}(G)(t) = \int_0^t G\, d^-W$. In general it seems to be unclear whether $\one_{[0,t]}G$ is forward integrable when $G$ is forward integrable.

First we show that the forward integral extends the It\^{o} integral in {\umd} spaces.
\begin{proposition}\label{prop:fwrd-extends-ito}
Consider an adapted process $G:[0,T]\times\OO \to \calL(H,X)$ that belongs to $L^0(\Omega;\gamma(0,T;H,X))$.
\begin{enumerate}[{\rm (1)}]
\item\label{it:Gn} For every $n\geq 1$, the process $G_n:= n\one_{[0,\frac1n]} * (\one_{[0,T]}P_n G)$ is adapted and in $L^0(\Omega;\gamma(0,T;H,X))$ and the following identity holds
\begin{align}\label{eq:identityIn}
I^{-}(G,n) = \int_0^\infty G_n \;dW = \int_0^{T+\frac1n} G_n \;dW.
\end{align}
\item\label{it:Gnforw} For every $t\in [0,T]$, $\one_{[0,t]}G$ is forward integrable and stochastically integrable and
\[J^{-}(G)(t) = \int_0^t G \;dW.\]
\end{enumerate}
\end{proposition}
Motivated by the above result, we will write $J(G)$ for $J^{-}(G)$ in the adapted case. Recall that $J(G)$ always has a continuous version and we will use this version without further notice. It is unclear to us whether $J^{-}(G,n)\to J(G)$ in $L^0(\OO;C([0,T];X))$ for all $G\in L^0(\Omega;\gamma(0,T;H,X))$. In the literature there are several attempts to prove such a result in the setting $H = X =\R$, but we could not follow these arguments.
In Theorem \ref{thrm:frwrd-ito-coincide-p>2} we will give sufficient conditions on $G$ for convergence in $L^0(\OO;W^{\alpha,p}([0,T];X))$ and in particular in $L^0(\OO;C^{\alpha-\frac1p}([0,T];X))$.

\begin{proof}
Choose an $H$-strongly progressively measurable version of $G$ and extend $G$ as zero on $(T,\infty)$.
Let the operator $S_n$ on $L^2(\RR_+;H)$ be given by $S_n f = n\one_{[0,\frac1n]} * P_n f$. Then $\|S_n\|\leq 1$ and it extends by \eqref{eq:gamma-ext} to a contraction on $\gamma(L^2(0,T;H),X)$. By duality and \eqref{eq:weakint}, this extension equals $R_{G_n}$. Hence $G_n$ is in $L^0(\OO;\gamma(\RR_+;H,X))$ and for every $t\in \RR_+$ and $x^*\in X^*$ one has
\[G_n(t)^*x^* = \int_{0}^{t} n \one_{[0,\frac1n]}(t-s) P_n G(s)^*x^*  \, ds\]
and since $G$ is progressively measurable, the latter is $\F_t$-measurable and thus $G_n$ is $H$-strongly adapted. It follows that $G_n$ is stochastically integrable and
by the stochastic Fubini theorem we obtain that for all $x^*\in X^*$,
\begin{align*}
\Big\lb\int_0^{T+\frac1n} G_n \;dW, x^*\Big\rb & = \int_0^\infty G_n^*x^* \;dW
\\ &= n\int_{0}^{\infty} \int_{0}^{T} \one_{[0,\frac1n]}(\sigma-s) P_n G(s)^*x^* \;ds\;dW(\sigma)
\\ & =  n\int_0^T \int_0^\infty \one_{[0,\frac1n]}(\sigma-s) P_n G(s)^*x^* \;dW(\sigma)\;ds
\\ & = n\sum_{k=1}^n \int_0^T \lb G(s)h_k, x^*\rb (W(s+1/n)-W(s))h_k\;ds
\\ & = \lb I^{-}(G,n), x^*\rb.
\end{align*}
By the Hahn-Banach theorem this yields \eqref{it:Gn}.

Next we prove \eqref{it:Gnforw}. Replacing $G$ by $\one_{[0,t]}G$  it suffices to consider $t=T$. Note that by \cite[Proposition 2.4]{NVW1}  $G_n\to G$ in $\gamma(\RR_+;H,X)$ pointwise on $\Omega$. Therefore, with \cite[Theorem 5.5]{NVW1}
we find that $I^{-}(G,n) = \int_0^\infty G_n\, d W\to \int_0^\infty G\, d W$ in $L^0(\Omega;X)$ and \eqref{it:Gnforw} follows.
\end{proof}

In the following lemma we collect some elementary properties of the forward integral.
\begin{lemma}\label{lemma:fwrd-int-props}
Let $X_0, X_1$ be Banach spaces and let $F, G: [0,T] \times \OO \to \calL(H,X_0)$ be forward integrable processes.
\begin{enumerate}[\rm (1)]
  \item For $\alpha,\beta\in \RR$, $\alpha F+\beta G$ is forward integrable, and a.s.
  \[ \int_0^T \alpha F + \beta G \;d^-W = \alpha\int_0^T F \;d^-W + \beta \int_0^T G \;d^-W.\]
  \item If $A: \OO \to \calL(X_0,X_1)$ is such that for every $x\in X_0$, $Ax$ is $\F$-measurable, then $A G$ is forward integrable and a.s.
  \[A \int_0^T G \;d^-W = \int_0^T A G \; d^-W.\]
  In particular, for any $x^* \in X^*_0$, $G^*x^*$ is forward integrable, and a.s.
  \[ \Big\lb \int_0^T G \;d^-W, x^* \Big\rb = \int_0^T G^*x^* \; d^-W.\]
  \item\label{eq:Hille} If $(A,D(A))$ is a closed linear operator on $X_0$ such that $G\in D(A)$ a.e., $A G$ is weakly in $L^2(0,T;H)$, $H$-strongly measurable and forward integrable, then  $\int_0^T G \;d^-W$ is in $D(A)$ and a.s.
  \[A \int_0^T G \;d^-W = \int_0^T A G\; d^-W.\]
\end{enumerate}
\end{lemma}

The property \eqref{eq:Hille} is a stochastic version of Hille's theorem (see \cite[Theorem II.6]{DU}). A version for the It\^o integral can be found in \cite[Lemma 2.8]{CoxGo}.

\begin{proof}
(1) and (2) are straightforward from the definition. To prove (3), note that by Hille's theorem,
\[A\int_0^T (G(s)h_k) (W(s+\tfrac1n)-W(s))h_k\;ds = \int_0^T A(G(s)h_k) (W(s+\tfrac1n)-W(s))h_k\;ds.\]
It follows that $A(I^-(G,n)) = I^-(AG,n) \to \int_0^T A G\;d^-W $ in probability. Also, $I^-(G,n) \to \int_0^T G \;dW $ in probability. Hence one can find a set $\OO_0\in \F$ with $\P(\Omega_0) = 1$ and a subsequence $(n_k)_{k\geq 1}$ such that for all $\omega\in \OO_0$, $I^-(G,n_k)(\omega) \to \Big(\int_0^T G \;d^{-}W \Big)(\omega)$ and $A(I^-(G,n_k)(\omega)) \to \Big( \int_0^T A G \;d^-W\Big)(\omega)$. Now the result follows from the assumption that $A$ is closed.
\end{proof}

Using the forward integral it is easy to deduce local properties of the stochastic integral.
\begin{remark}\label{rmrk:fwrd-int-props}\
From Lemma \ref{lemma:fwrd-int-props} it follows that for a forward integrable process $G$ and a set $B\in \mathscr{F}$, $\one_B G$ is forward integrable and
\[ \int_0^T \one_B G \;d^-W = \one_B \int_0^t G \;d^-W.\]
In particular, if $G\in L^0(\OO;\gamma(0,T;H,X))$ is adapted and for all $x^*\in X^*$, $G^*x^* = 0$ on a set $(0,T)\times B$, then a.s.
\[ 0=\int_0^t \one_B G^*x^* \;d^-W = \Big\lb \one_B \int_0^t G \;d^-W, x^*\Big\rb, \ \ \ x^*\in X^*, t\in [0,T].\]
In particular, we deduce that $\int_0^\cdot G \;d^-W =0$ on $B$ a.s.
\end{remark}

\section{Convergence and path regularity\label{sec:conv}}

In this section we will give conditions under which for adapted $G$ one has $J^{-}(G,n)\to J(G)$ in the Sobolev norm.
Before we start we introduce a class of functions.

\begin{definition}
For $\beta\in [0,\tfrac12)$ and $p\in [1, \infty)$, let $V^{\beta,p}(0,T;H,X)$ denote the space of $H$-strongly measurable $G:(0,T)\to \calL(H,X)$ for which for almost all $t\in [0,T]$, $r\mapsto (t-r)^{-\beta} G(r)$ is in $\gamma(0,t;H,X)$ and
\[\|G\|_{V^{\beta,p}(0,T;H,X)} := \Big( \int_0^T \|r\mapsto (t-r)^{-\beta} G(r)\|_{\gamma(0,t;H,X)}^p\, dt \Big)^{1/p} <\infty.\]
\end{definition}
The spaces $V^{\beta,p}(0,T;H,X)$ were introduced in \cite{NVW-evolution} in order to study stochastic evolution equations of semilinear type in {\umd} spaces. They also play a major role in \cite{CvN} and \cite{KvN11}, where results on approximation of SPDEs have been derived. Although the spaces $V^{\alpha,p}$ look rather involved at first sight they are quite useful and not too difficult to work with. Many properties of Bochner spaces are inherited by the spaces $V^{\beta,p}(0,T;H,X)$. The main motivation for the weight inside the $\gamma$-norm is that it increases the integrability properties of $G$ without leaving the $\g$-setting.

In this paper the spaces $V^{\beta, p}(0,T;H,X)$ play an important role. They allow us to prove the main results for all {\umd} Banach spaces $X$. If the Banach space $X$ also has type $2$, then the spaces $V^{\beta, p}(0,T;H,X)$ may be replaced by $L^p(0,T;\g(H,X))$ due to Proposition \ref{prop:embeddingtype}(3). The main results in this situation are stated in Corollary \ref{cor:type2prop}, Corollary \ref{cor:type2} and Corollary \ref{cor:productforward}.

The following embedding results are straightforward from the definition and \eqref{eq:Linftyproduct}
\begin{align*}
V^{\beta,p_0}(0,T;H,X)& \hookrightarrow V^{\beta,p_1}(0,T;H,X) \ \ \text{if $1\leq p_1<p_0<\infty$},
\\ V^{\beta_0,p}(0,T;H,X)& \hookrightarrow V^{\beta_1,p}(0,T;H,X) \ \ \text{if $0\leq \beta_1<\beta_0<\frac12$}.
\end{align*}

The next proposition gives several other embedding properties for the spaces $V^{\beta, p}(0,T;H,X)$. In particular they give new insights for results in \cite{CvN}, \cite{KvN11} and \cite{NVW-evolution}. Details on (co)type properties of a Banach space can be found in \cite[Chapter 11]{DJT}. Recall that every Hilbert space has type $2$, and $X = L^q$ (or $X = W^{s,q}$) has type $2$ if and only if $q\in [2, \infty)$. Moreover, for $q<\infty$, $L^q$ has cotype $q\vee 2$.

\begin{proposition}\label{prop:embeddingtype}
Let $p\geq 1$ and $\beta\in [0,\tfrac12)$.
\begin{enumerate}[{\rm (1)}]
\item\label{it:trvialemb} If $G\in V^{\beta,1}(0,T;H,X)$, then for all $\varepsilon \in (0,T)$ one has $G\in \g(0,T-\varepsilon;H,X)$ and
\[ \|G\|_{\g(0,T-\varepsilon;H,X)} \leq \frac{T^{\beta}}{\varepsilon}  \|G\|_{V^{\beta,1}(0,T;H,X)}.\]
Moreover, if $\beta>\frac1p$, then
\[V^{\beta,p}(0,T;H,X)\hookrightarrow \g(0,T;H,X).\]
\item\label{it:typeemb}  If $X$ has cotype $p\in [2, \infty)$ and $\beta\in[0,\frac1p)$, then
\begin{align*}
\g(0,T;H,X) \hookrightarrow  V^{\beta,p}(0,T;H,X).
\end{align*}
\item\label{it:cotypeemb} If $X$ has type $2$ and $p\in [2, \infty)$, then
\begin{align*}
L^p(0,T;\gamma(H,X)) \hookrightarrow V^{\beta,p}(0,T;H,X).
\end{align*}
\end{enumerate}
\end{proposition}
Under type $p$ assumptions one can show that $V^{\beta,p}(0,T;H,X)$ contains certain fractional Sobolev spaces or H\"older spaces, but we will not go into details on this (see \cite[Lemma 3.3]{NVW-evolution} and \cite[Lemma 3.8]{KvN11} for some details in this direction).

Note that for fixed $G\in V^{\beta,p}(0,T;H,X)$, the function $u\mapsto \one_{[0,u]}G$ is continuous from $[0,T]$ into $V^{\beta,p}(0,T;H,X)$ (see \cite[Section 7]{NVW-evolution}).

\begin{proof}
\eqref{it:trvialemb}:
For every $s\in [0,T)$, we can write
\[G(s) = \int_0^T (t-s)^{-\beta} G(s) \frac{(t-s)^{\beta}}{T-s} \one_{[s,T]}(t) \, dt.\]
It follows that for every $\varepsilon\in [0,T]$ one has
\begin{equation}\label{eq:estGeps}
\|G\|_{\g(0,T-\varepsilon;H,X)} \leq \int_0^T \Big\|s\mapsto (t-s)^{-\beta} G(s) \frac{(t-s)^{\beta}}{T-s} \one_{[s,T]}(t)\Big\|_{\g(0,T-\varepsilon;H,X)} \, dt.
\end{equation}
For all $\varepsilon\in (0,T)$ and $s\in [0,T-\e)$, $\tfrac{(t-s)^{\beta}}{T-s} \leq \varepsilon^{-1} T^{\beta}$, and thus by \eqref{eq:Linftyproduct}
\[\|G\|_{\g(0,T-\varepsilon;H,X)} \leq \varepsilon^{-1} T^{\beta} \|G\|_{V^{\beta,1}(0,T;H,X)}.\]

Next assume $\beta>\frac1p$ and take $\varepsilon = 0$ in \eqref{eq:estGeps}.
Note that for all $t\in [0,T)$ and $s\in [0,t]$, $\tfrac{(t-s)^{\beta}}{T-s}\leq (T-s)^{\beta-1} \leq (T-t)^{\beta-1}$. Therefore, by \eqref{eq:Linftyproduct}, and H\"older's inequality,
\begin{align*}
\|G\|_{\g(0,T;H,X)} &\leq \int_0^T \|s\mapsto (t-s)^{-\beta} G(s) \one_{[0,t]}(s)\|_{\g(0,T;H,X)} (T-t)^{\beta-1} \, dt
\\ & \leq \|G\|_{V^{\beta,p}(0,T;H,X)} \Big(\int_0^T  (T-t)^{(\beta-1)p'} \, dt\Big)^{1/p'}
\\ & \leq C \|G\|_{V^{\beta,p}(0,T;H,X)} .
\end{align*}

\eqref{it:typeemb}: \ Let $G\in \g(0,T;H,X)$. Let $\phi_t(r) = \one_{(0,t)}(r) (t-r)^{-\beta}$ and $M_{\beta}:(0,T)\to \calL(X, L^p(0,T;X))$ be given by $M_\beta(t) x = \phi_t x$. Observe that by the $\gamma$-Fubini isomorphism (see \cite[Proposition 2.6]{NVW1}) and the definition of $V^{\beta,p}$
\begin{align}\label{eq:VMG}
c^{-1}\|G\|_{V^{\beta,p}(0,T;H,X)} \leq \|M_{\beta} G\|_{\g(0,T;H,L^p(0,T;X))} \leq c \|G\|_{V^{\beta,p}(0,T;H,X)}.
\end{align}
For $\beta<\frac1p$ and $t\in (0,T)$, one has
\begin{align*}
K:&=\int_0^\infty \sup_{t\in (0,T)} \mu\big(\{r\in (0,t):\phi_t(r)>s\}\big)^{1/p} \, ds \\
 &= \int_0^\infty \sup_{t\in (0,T)} (t\wedge s^{-\frac{1}{\beta}})^{1/p}\,ds = \int_0^\infty T^{\frac1p}\wedge s^{-\frac{1}{\beta p}} \, ds <\infty.
\end{align*}
Therefore, it follows from \cite[Lemma 3.1]{HytVer} that $\{M_{\beta}(t):t\in (0,T)\}$ is $R$-bounded by $CK$, and hence by the Kalton--Weis $\gamma$-multiplier theorem (see \cite[Theorem 5.2]{Neerven-Radon}), we find that
\begin{align*}
\|M_{\beta} G\|_{\g(0,T;H,L^p(0,T;X))} \leq C K \|G\|_{\g(0,T;H,X)},
\end{align*}
where we used the fact that $L^p(0,T;X)$ does not contain a copy of $c_0$ as it has finite cotype (see \cite[page 212 and Theorem 11.12]{DJT}). Combining the latter  estimate with \eqref{eq:VMG}, the required result follows.

\eqref{it:cotypeemb}: \ From $L^2(0,T;\g(H,X))\hookrightarrow \g(0,T;H,X)$ (see \cite[Theorem 11.6]{Neerven-Radon}) and Young's inequality for convolutions we obtain
\begin{align*}
\|G\|_{V^{\beta,p}(0,T;H,X)}^p & = \int_0^T \|r\mapsto (t-r)^{-\beta} G(r)\|_{\g(0,t;H,X)}^p \, dt
\\ & = C \int_0^T  \Big(\int_0^t (t-r)^{-2\beta} \|G(r)\|_{\g(H,X)}^2\, dr\Big)^{p/2} \, dt
\\ & \leq C \Big(\int_0^T r^{-2\beta} \,dr\Big)^{p/2}  \int_0^T \|G(r)\|_{\g(H,X)}^p\, dr
\\ & = C' \|G\|_{L^p(0,T;\gamma(H,X))}^p.
\end{align*}
\end{proof}

\begin{example}\label{ex:HScase}
Let $X$ be a Hilbert space. In the case that $p=2$ and $\beta\in [0,\tfrac12)$, by \eqref{eq:HSid} and Fubini's theorem, one has
\begin{equation}\label{eq:locallyL^2}
V^{\beta,2}(0,T;H,X) = L^2((0,T),\mu_{\alpha,T} ;\mathcal{C}^2(H,X))
\end{equation}
where and $d\mu_{\alpha,T}(r) = (T-r)^{1-2\beta} \, dr$. Moreover, by Proposition \ref{prop:embeddingtype} one has
\[L^p(0,T;\gamma(H,X)) \hookrightarrow V^{\beta,p}(0,T;H,X) \ \ \ \text{for all $p\geq 2$ and $\beta\in [0,\tfrac12)$},\]
\[L^2(0,T;\mathcal{C}^2(H,X)) \hookrightarrow V^{\beta,p}(0,T;H,X) \ \ \ \text{for all $p\geq 2$ and $\beta\in [0,\tfrac1p)$},\]
\[V^{\beta,p}(0,T;H,X) \hookrightarrow L^2(0,T;\mathcal{C}^2(H,X)) \ \ \ \text{for all $p\geq 2$ and $\beta\in (\tfrac1p,\tfrac12)$}.\]
\end{example}

Next we prove a pathwise regularity result for $J(G)$. Recall that $J(G)$ is the process given by $J(G)(t) = \int_0^t G\, dW$.

\begin{proposition}\label{prop:M_Holder_2}
Let $p \in [1, \infty)$ and $0<\alpha<\beta<\frac12$. If $G$ is an adapted process that belongs to $L^0(\OO;V^{\beta,p}(0,T;H,X))$, then $J(G)\in L^0(\OO;W^{\alpha,p}(0,T; X))$. Furthermore, the following assertions hold:
\begin{enumerate}[{\rm (1)}]
\item\label{it:LpestV}
There exists a constant $C$ independent of $G$ such that
\[\|J(G)\|_{L^p(\OO;W^{\alpha,p}(0,T; X))} \leq C \|G\|_{L^p(\OO;V^{\beta,p}(0,T;H,X))}.\]
\item\label{it:L0estV} For every $n\geq 1$, assume that $G_n \in L^0(\OO;V^{\beta,p}(0,T;H,X))$ is an adapted process.
If $G_n\to G$ in $L^0(\OO;V^{\beta,p}(0,T;H,X))$, then
\[J(G_n) \to J(G) \ \ \text{in $L^0(\OO;W^{\alpha,p}(0,T; X))$}.\]
\end{enumerate}
\end{proposition}

\begin{remark}\label{rem:welldef}
Note that under the above assumptions, by Proposition \ref{prop:embeddingtype} \eqref{it:trvialemb} one has $G\one_{[0,t]}\in L^0(\OO;\gamma(0,T;H,X))$ for all $t\in [0,T)$, and therefore, $J(G)(t)$ is well-defined for every $t\in [0,T)$.
\end{remark}

\begin{remark}\label{rmrk:sblv-hldr-embdng_2}
If $\frac{1}{p}<\alpha<\frac12$, we can use the Sobolev embedding theorem \eqref{eq:fractionalSobolevemb}, to replace $W^{\alpha,p}(0,T;X)$ by $C^{\alpha-\frac1p}(0,T;X)$ in the above result.
\end{remark}

\begin{example}
Let $X$ be a Hilbert space. From Example \ref{ex:HScase}, we see that by \eqref{eq:locallyL^2} and Proposition \ref{prop:M_Holder_2}, for every
$G\in L^0(\OO;L^2((0,T),\mu_{\alpha,T};\mathcal{C}^2(H,X)))$ adapted, one has $J(G)\in L^0(\OO;W^{\alpha,p}(0,T;X))$. Note that such a process $G$ is not necessarily in $L^0(\OO;L^2(0,T;\mathcal{C}^2(H,X)))$. In the case $H = X =\R$ an example is given by $G(t) = (T-t)^{-\frac12-\varepsilon}$ with $\varepsilon>0$.

Indeed, one easily checks that $G\in L^2((0,T),\mu_{\alpha,T})$ if and only if $\varepsilon+\alpha<1/2$, and in that case $J(G)\in W^{\alpha,p}(0,T)$ a.s. However, $G\notin L^2(0,T)$. This singular behavior can only occur at the point $t=T$ as follows from Proposition \ref{prop:embeddingtype} \eqref{it:trvialemb}.
\end{example}

\begin{proof}[Proof of Proposition \ref{prop:M_Holder_2}]
To prove \eqref{it:LpestV}, note that for $0\leq s\leq r<t\leq T$, one has $1\leq (t-s)^{\beta} (t-r)^{-\beta}$, and hence by \eqref{eq:Itoisom} and \eqref{eq:Linftyproduct}, we have
\begin{align}
\nonumber \E \|J(G)(t) - J(G)(s)\|^p & \leq C \E \|G\|_{\gamma(s,t;H,X)}^p
\\ & \label{eq:incremJst}\leq C(t-s)^{\beta p} \E \|r\mapsto (t-r)^{-\beta} G(r)\|_{\gamma(s,t;H,X)}^p
\\ & \nonumber \leq C (t-s)^{\beta p} \E \|r\mapsto (t-r)^{-\beta} G(r)\|_{\gamma(0,t;H,X)}^p.
\end{align}
By Fubini's theorem we find that
\begin{align}
\nonumber \E[J(G)]^p_{W^{\alpha,p}(0,T;X)} & =
2\int_0^T \int_0^t \frac{\E \|J(G)(t) - J(G)(s)\|^p}{(t-s)^{\alpha p+1}}\;ds\;dt
\\ & \label{eq:Walphaestimate}  \leq C \E \int_0^T \int_0^t \frac{\|r\mapsto (t-r)^{-\beta} G(r)\|^p_{\gamma(0,t;H,X)}}{(t-s)^{1-(\beta-\alpha) p}}\;ds\;dt
\\ & \nonumber \leq C T^{(\beta-\alpha)p} \E \int_0^T \|r\mapsto (t-r)^{-\beta} G(r)\|^p_{\gamma(0,t;H,X)} \;dt
\\ & \nonumber = C T^{(\beta-\alpha)p} \|G\|_{L^p(\OO;V^{\beta,p}(0,T;H,X))}^p,
\end{align}
where we used $\beta>\alpha$. Taking $s=0$ in \eqref{eq:incremJst}, one also obtains
\begin{align*}
\E\|J(G)\|^p_{L^p(0,T;X)} & \leq C T^{\beta p} \E  \int_0^T  \|r\mapsto (t-r)^{-\beta} G(r)\|_{\g(0,t;H,X)}^p \, dt \\ & =  T^{\beta p} \|G\|_{L^p(\OO;V^{\beta,p}(0,T;H,X))}^p.
\end{align*}
Combining the estimates yields that $J(G)\in L^p(\OO;W^{\alpha,p}(0,T; X))$ and \eqref{it:LpestV} holds.

Before we continue to the proof of \eqref{it:L0estV}, we first prove that for adapted processes $G\in L^0(\OO;V^{\beta,p}(0,T;H,X))$, one has $J(G)\in L^0(\OO;W^{\alpha,p}(0,T; X))$. Let $\tau_n$ be the stopping time given by
\[\tau_n = \inf\{t\in[0,T]:\|\one_{[0,t]} G\|_{V^{\beta,p}(0,T;H,X)}\geq n\},\]
where we put $\tau_n = T$ if the infimum is taken over the empty set.
Then $\one_{[0,\tau_n]} G\in L^p(\OO;V^{\beta,p}(0,T;H,X))$ and hence $t\mapsto J(G)(t\wedge\tau_n) = J(\one_{[0,\tau_n]} G)(t)$ belongs to $L^0(\OO;W^{\alpha,p}(0,T; X))$. Since for almost every $\omega\in\OO$, we can find an $n\geq 1$ with $\tau_n(\omega) = T$, we obtain $J(G)\in W^{\alpha,p}(0,T; X)$ almost surely.

To prove \eqref{it:L0estV} we use another stopping time argument. By linearity we can replace $G_n$ by $G_n - G$ and hence it suffices to consider $G=0$. Moreover, by a subsequence argument it suffices to consider the case that $G_n\to 0$ in $V^{\beta,p}(0,T;H,X)$ almost surely.
For $n\geq 1$ let $\tau_n$ be the stopping time given by
\[\tau_n = \inf\{s\in[0,T]:  \|\one_{[0,s]} G_n\|_{V^{\beta,p}(0,T;H,X)}\geq 1\}.\]
Since $G_n\to 0$ in $V^{\beta,p}(0,T;H,X)$ almost surely, we find that $\limn \P(\tau_n = T) = 1$.
Since $\|\one_{[0,\tau_n]}G_n\|_{V^{\beta,p}(0,T;H,X)}\leq 1$, and
\[\|\one_{[0,\tau_n]}G_n\|_{V^{\beta,p}(0,T;H,X)}\leq \|G_n\|_{V^{\beta,p}(0,T;H,X)}\to 0 \ \text{a.s.},\]
the dominated convergence theorem gives that \[\one_{[0,\tau_n]}G_n\to 0\] in the space $L^p(\OO;V^{\beta,p}(0,T;H,X))$.
In particular, by \eqref{it:LpestV} one has $J(\one_{[0,\tau_n]}G_n)\to 0$ in $L^p(\OO;W^{\alpha,p}(0,T;X))$.
Let $\varepsilon>0$ be arbitrary. Then using $J(\one_{[0,\tau_n]}G_n)(t) = J(G_n)(t\wedge \tau_n)$ we find that
\begin{align*}
\P\big(\|J(G_n)\|_{W^{\alpha,p}(0,T;X)}\geq \varepsilon\big) & \leq
\P\big(\|J(G_n)\|_{W^{\alpha,p}(0,T;X)}\geq \varepsilon, \tau_n = T\big) + \P(\tau_n<T)
\\ & \leq \P\big(\|J(\one_{[0,\tau_n]} G_n)\|_{W^{\alpha,p}(0,T;X)}\geq \varepsilon\big) + \P(\tau_n<T)
\\ & \leq \varepsilon^{-p} \E\|J(\one_{[0,\tau_n]} G_n)\|_{W^{\alpha,p}(0,T;X)}^p + \P(\tau_n<T).
\end{align*}
Now the result follows by letting $n\to \infty$.
\end{proof}

The next result is one of the main results of the paper and gives convergence of paths of the forward integral in Sobolev norms. With Remark \ref{rmrk:sblv-hldr-embdng_2} one can derive convergence in the H\"older norm as a consequence.
\begin{theorem}\label{thrm:frwrd-ito-coincide-p>2}
Let $p \in [1, \infty)$ and $0<\alpha<\beta<\frac12$.
\begin{enumerate}[{\rm (1)}]
\item\label{it:forwardconvWLp} If $G\in L^p(\OO;V^{\beta,p}(0,T;H,X))$ is adapted, then
\[J^{-}(G,n)\to J(G) \ \ \text{in} \ \ \  L^p(\Omega; W^{\alpha,p}(0,T;X)).\]
\item\label{it:forwardconvWL0} If $G\in L^0(\OO;V^{\beta,p}(0,T;H,X))$ is adapted, then
\[J^{-}(G,n)\to J(G) \ \ \text{in} \ \ \  L^0(\Omega; W^{\alpha,p}(0,T;X)).\]
\end{enumerate}
\end{theorem}
From Remark \ref{rem:welldef} we see that $J(G)(t)$ and $J^{-}(G,n)(t)$ are well-defined for every $t\in [0,T)$.

Recall from \eqref{eq:identityIn} that
\[J^{-}(G,n)(t) = \int_0^{\infty} G_n \, d W, \ \ \ \text{where} \ \ \ \ G_n = n\one_{[0,\frac1n]} * (\one_{[0,t]}P_n G).\]
Since $G_n\to G$ in $L^0(\OO;V^{\beta,p}(0,T;H,X)$, at first sight it seems that Proposition \ref{prop:M_Holder_2} can be used directly to obtain Theorem \ref{thrm:frwrd-ito-coincide-p>2}. Unfortunately, Proposition \ref{prop:M_Holder_2} does not apply because the process $G_n$ also depends on $t$, and we need to proceed differently.

\begin{proof}
Before proving the assertion we note that if $G\in L^0(\OO;V^{\beta,p}(0,T;H,X))$, then $J(G)\in W^{\alpha,p}(0,T;X)$ a.s.\ by Proposition \ref{prop:M_Holder_2}. We claim that $J^{-}(G,n)\in W^{\alpha,p}(0,T;X)$ a.s.
Indeed,
\begin{align*}
\|J^{-}&(G,n)(t) - J^{-}(G,n)(s)\| \\&= \Big\|\sum_{k=1}^n n\int_s^t G(r) h_k (W(r+1/n)h_k - W(r)h_k)\;dr\Big\|
\\ & \leq \sum_{k=1}^n n \Big\|\int_0^T G(r) h_k \one_{[s,t]}(r)(W(r+1/n)h_k - W(r)h_k)\;dr\Big\|=:\sum_{k=1}^n n J_k.
\end{align*}
By \eqref{eq:gamma-estm} we find that
\begin{align*}
J_k &\leq \|r\mapsto (t-r)^{-\beta} \one_{[0,r]} G(r)\|_{\g(0,t;H,X)}\\&\qquad \qquad \times \|r\mapsto (t-r)^{\beta} (W(r+1/n)h_k - W(r)h_k)\|_{L^2(s,t)}.
\end{align*}
Since the paths of $r\mapsto W(r+1/n)h_k - W(r)h_k$ are a.s. bounded, we have
\[\|r\mapsto (t-r)^{\beta} (W(r+1/n)h_k - W(r)h_k)\|_{L^2(s,t)} \leq C(W,n) (t-s)^{\beta+\frac12},\]
where $\displaystyle C(W,n) = 2\sup_{r\in [0,T+1]} \sup_{1\leq k\leq n} |W(r)h_k|.$ It follows that
\begin{align*}
[J^{-}&(G,n)]_{W^{\alpha,p}(0,T;X)}^p
 = 2 \int_0^T \int_0^t \frac{\|J^{-}(G,n)(t) - J^{-}(G,n)(s)\|^p}{(t-s)^{\alpha p +1 }} \, ds \, dt
\\ & 2\leq C_{W,n} \int_0^T \int_0^t \|r\mapsto (t-r)^{-\beta} \one_{[0,r]} G(r)\|_{\g(0,t;H,X)}^p (t-s)^{(\beta-\alpha + \frac12)p-1} \, ds \, dt
\\ & \leq C_{W,n} C\|G\|_{V^{\beta,p}(0,T;X)}^p.
\end{align*}
Similarly, one sees that $\|J^{-}(G,n)\|_{L^p(0,T;X)}<\infty$ a.s.\ and the claim follows.

\eqref{it:forwardconvWLp}: Observe that by \eqref{eq:identityIn} and \eqref{eq:Itoisom},
\begin{equation}\label{eq:Jminuniform}
\begin{aligned}
\EE&[J^{-}(G,n)  - J(G)]^p_{W^{\alpha,p}(0,T;X))}
\\ &
\leq  C \E \int_0^T \int_0^T \one_{[0,t]}(s)\frac{\|n\one_{[0,\frac1n]} * (\one_{[s,t]}P_n G) - \one_{[s,t]}G\|_{\g(\R_+;H,X)}^p}{|t-s|^{\alpha p+1}} \, ds \, dt.
\end{aligned}
\end{equation}
We will use the dominated convergence theorem to show that the latter converges to zero as $n\to\infty$. Indeed, by Young's inequality one has $\|n\one_{[0,\frac1n]} * f\|_{L^2(\R;H)}\leq \|f\|_{L^2(\R;H)}$ for $f\in L^2(\R;H)$.
Therefore, by the right-ideal property and \eqref{eq:Linftyproduct} for $0\leq s\leq t\leq T$,
\begin{align*}
\|n\one_{[0,\frac1n]} * (\one_{[s,t]}P_n G) - \one_{[s,t]}G\|_{\g(\R_+;H,X)}& \leq 2 \E \|\one_{[s,t]}G\|_{\g(s,t;H,X)}
\\ & \leq 2|t-s|^{\beta} \|r\mapsto (t-r)^{-\beta} G(r)\|_{\g(0,t;H,X)}.
\end{align*}
Now the latter is integrable on the space $\OO\times [0,T]^2$ with measure $\one_{[0,t]}(s) (t-s)^{-\alpha p-1}\, ds\, dt\, d\P$, and it dominates the function $\one_{[0,t]}(s) \|n\one_{[0,\frac1n]} * (\one_{[s,t]}P_n G) - \one_{[s,t]}G\|_{\g(\R_+;H,X)}^p$, which depends on $0\leq s\leq t\leq T$ and $\omega\in \OO$.
Moreover, by \cite[Proposition 2.4]{NVW1}
\begin{align*}
\limn \|n\one_{[0,\frac1n]} * (\one_{[s,t]}P_n G) - \one_{[s,t]}G\|_{\g(\R_+;H,X)} = 0
\end{align*}
for all $0\leq s\leq t\leq T$ and a.s.\ on $\OO$.
Therefore, by the dominated convergence theorem, the right-hand side of \eqref{eq:Jminuniform} tends to zero as $n\to \infty$.

A similar argument yields that $\EE\|J^-(G,n) - J(G)\|^p_{L^p(0,T;X)} \to 0$ as $n\to\infty$. This proves \eqref{it:forwardconvWLp}.

Next we prove \eqref{it:forwardconvWL0} using a stopping time argument. Consider an element $G\in L^0(\OO;V^{\beta,p}(0,T;H,X))$.
For each $m\geq 1$ define
\[\tau_m = \inf\{[0,T]: \|\one_{[0,t]}G\|_{V^{\beta,p}(0,T;H,X)}\geq m\},\]
where we let $\tau_m = T$ if the infimum is taken over the empty set.
Let $G_{m} = \one_{[0,\tau_m]} G$. Clearly, $\lim_{m\to\infty} \P(\tau_m = T) = 1$.
Observe that almost surely, for all $t\in [0,T]$, $J(G)(\tau_m\wedge t) = J(\one_{[0,\tau_m]} G)(t)$ and  $J^{-}(G,n)(\tau_m\wedge t) = J^{-}(\one_{[0,\tau_m]} G,n)(t)$. The latter is trivial as $J^{-}(\cdot, n)$ is defined in a pathwise sense.

Let $\varepsilon>0$ and $\delta>0$ be arbitrary and choose $m$ so large that $\P(\tau_m <T) <\delta$. It follows that for all $n\geq 1$,
\begin{align*}
\P\big(&\|J(G)- J^{-}(G,n)\|_{W^{\alpha,p}(0,T;X)}\geq \varepsilon\big)
\\ & \leq  \P\big(\|J(G)- J^{-}(G,n)\|_{W^{\alpha,p}(0,T;X)}\geq \varepsilon, \tau_m = T\big) + \P(\tau_m<T)
\\ & \leq \P\big(\|J(\one_{[0,\tau_m]}G)- J^{-}(\one_{[0,\tau_m]}G ,n)\|_{W^{\alpha,p}(0,T;X)}\geq \varepsilon\big) + \delta
\\ & \leq \varepsilon^{-p} \E\|J(\one_{[0,\tau_m]}G)- J^{-}(\one_{[0,\tau_m]}G ,n)\|_{W^{\alpha,p}(0,T;X)}^p  + \delta.
\end{align*}
Since $\one_{[0,\tau_m]}G$  satisfies the conditions of \eqref{it:forwardconvWLp} it follows that
\[\limsup_{n\to \infty} \P\big(\|J(G)- J^{-}(G,n)\|_{W^{\alpha,p}(0,T;X)}\geq \varepsilon\big) \leq \delta.\]
Since $\delta>0$ was arbitrary, the result follows.
\end{proof}

If the space $X$ is not only a {\umd} space, but has type $2$ as well, then one can obtain further conditions for a process to be in the spaces considered in Theorem \ref{thrm:frwrd-ito-coincide-p>2}.
Both results below follow immediately from the embedding of Proposition \ref{prop:embeddingtype} \eqref{it:typeemb}, Proposition \ref{prop:M_Holder_2} and Theorem \ref{thrm:frwrd-ito-coincide-p>2}. Similar corollaries can be deduced from Proposition \ref{prop:embeddingtype} \eqref{it:cotypeemb}.

\begin{corollary}\label{cor:type2prop}
Assume $X$ has type $2$, and let $p \in [2, \infty)$ and $0<\alpha<\frac12$.
If $G \in L^0(\OO;L^p(0,T;\g(H,X)))$ is adapted, then $J(G)\in L^0(\OO;W^{\alpha,p}(0,T; X))$. Furthermore, the following assertions hold:
\begin{enumerate}[{\rm (1)}]
\item\label{it:LpestVtype2}
There exists a constant $C$ independent of $G$ such that
\[\|J(G)\|_{L^p(\OO;W^{\alpha,p}(0,T; X))} \leq C \|G\|_{L^p(\OO;L^p(0,T;\g(H,X)))}.\]
\item\label{it:L0estVtype2} Assume that for every $n\geq 1$, $G_n \in L^0(\OO;L^p(0,T;\g(H,X)))$ is an adapted process. If $G_n\to G$ in $L^0(\OO;L^p(0,T;\g(H,X)))$, then
\[J(G_n) \to J(G) \ \ \text{in $L^0(\OO;W^{\alpha,p}(0,T; X))$}.\]
\end{enumerate}
\end{corollary}

\begin{corollary}\label{cor:type2}
Assume $X$ has type $2$, and let $p \in [2, \infty)$.
\begin{enumerate}[{\rm (1)}]
\item\label{it:forwardconvWLp2} If $G\in L^p(\OO; L^p(0,T;\gamma(H,X)))$ is adapted, then for all $\alpha\in (0,\tfrac12)$,
\[J^{-}(G,n)\to J(G) \ \ \ \text{in $L^p(\Omega; W^{\alpha,p}(0,T;X))$.}\]
\item\label{it:forwardconvWL02} If $G\in L^0(\OO;L^p(0,T;\g(H,X)))$ is adapted, then for all $\alpha\in (0,\tfrac12)$,
\[J^{-}(G,n)\to J(G) \ \ \  \text{in $L^0(\Omega; W^{\alpha,p}(0,T;X))$}.\]
\end{enumerate}
\end{corollary}
Again, Remark \ref{rmrk:sblv-hldr-embdng_2} applies to the above results and this will give convergence in the H\"older norm. The above result contains Theorem \ref{thm:intro} as a special case.

\section{Nonadapted pointwise multipliers\label{sec:nonadaptedmult}}

In the next result we give sufficient smoothness conditions on a possibly non-adapted operator-valued process $M$ and an adapted process $G$, such that $MG$ becomes forward integrable. Moreover we derive a neat integration by parts formula which yields a very useful representation formula for the forward integral. Recall that $I(G) = \int_0^T G\; dW$.

\begin{theorem}\label{thm:productforward}
Let $X$ and $Y$ be {\umd} Banach spaces. Assume $p\in (2, \infty)$, $\delta \in [0,3/2)$ and $\beta\in (\tfrac1p,\tfrac12)$ are such that $\beta-\frac1p-\delta+1>0$.
Let $M:[0,T]\times\OO\to \calL(X,Y)$ be such that
\begin{enumerate}[(i)]
\item For all $x\in X$, $(t, \omega)\mapsto M(t,\omega) x$ is strongly measurable.
\item For almost all $\omega\in\OO$, $t\mapsto M(t,\omega)$ is continuously differentiable on $[0,T)$ and there exists a constant $\delta\in [0,\tfrac32)$ such that for almost all $\omega\in\OO$, there is a constant $C(\omega)>0$ such that
\begin{align*}
\|M'(t,\omega)\| \leq  C(\omega)(T-t)^{-\delta}, \ \ t\in [0,T).
\end{align*}
\end{enumerate}
Assume $G\in L^0(\OO;V^{\beta,p}(0,T;X))$ is adapted and $MG$ is weakly in $L^2(0,T;H)$.
Then $M G$ is forward integrable, $s\mapsto M'(s) I(\one_{[s,T]}G)\in L^1(0,T;Y)$ almost surely and
\begin{equation}\label{eq:forwardintbyparts}
\int_0^T M(s) G(s)\, d^{-}W(s) = M(0) I(G) + \int_0^T M'(s) I(\one_{[s,T]}G) \, ds.
\end{equation}
\end{theorem}

Note that we do not assume any adaptedness properties on $M$.

\begin{proof}
By Proposition \ref{prop:embeddingtype}, \eqref{it:trvialemb} $G\in L^0(\OO;\gamma(0,T;H,X))$.

Fix $t\in (0,T)$. Let $f_k = n G(\cdot) h_k (W(\cdot+1/n)h_k - W(\cdot)h_k)$. Note that by \eqref{eq:Linftyproduct} and the path continuity of $W h_k$, we have $f_k\in L^0(\OO;\g(0,t;X))$. Let $F_k:[0,t]\times\OO\to X$ be given by $F_k(s) = \int_s^t f_k(r) \, dr$ and note that
\[\sum_{k=1}^n F_k(s) = I^{-}(\one_{[s,t]}G,n).\]
Fix $\omega\in \OO$. By Lemma \ref{lem:intbyparts} both $M G$ and $M f_k$ are in $\g(0,t;H,Y)$ and
\begin{align*}
I^{-}(M \one_{[0,t]}G,n) & = \sum_{k=1}^n \int_0^t M(s) f_k(s) \,ds = \sum_{k=1}^n M(0)F_k(0) + \int_0^t M'(s) F_k(s) \,ds
\\ & = M(0)I^{-}(\one_{[0,t]}G,n) + \int_0^t M'(s) I^{-}(\one_{[s,t]}G,n) \,ds.
\end{align*}
Now letting $t\uparrow T$, it follows from the observation below \eqref{eq:J-def} that
\[M(0)I^{-}(\one_{[0,t]}G,n)\to M(0)I^{-}(G,n) \ \ \text{and} \ I^{-}(M \one_{[0,t]}G,n)\to I^{-}(MG,n).\]
Next we claim that for $t\uparrow T$,
\begin{align}\label{eq:foreachnM}
\int_0^t M'(s) I^{-}(\one_{[s,t]}G,n) \,ds \to \int_0^T M'(s) I^{-}(\one_{[s,T]}G,n) \,ds.
\end{align}
Indeed, choose $\alpha\in (\tfrac1p,\beta)$ such that $\alpha-\frac1p-\delta+1>0$.
Note that by Theorem \ref{thrm:frwrd-ito-coincide-p>2} and \eqref{eq:fractionalSobolevemb}, $K:=\|J^{-}(G,n)\|_{C^{\alpha-\frac1p}(0,T;X)}<\infty$ for almost all $\omega\in\OO$. The difference of both of the terms in \eqref{eq:foreachnM} can be estimated by
\begin{align*}
\int_t^T & \|M'(s) I^{-}(\one_{[s,T]}G,n)\| \,ds + \int_0^t \|M'(s) I^{-}(\one_{[t,T]}G,n)\| \,ds
\\ & \leq \int_t^T \|M'(s) (J^{-}(G,n)(T) - J^{-}(G,n)(s)) \| \,ds
\\ &\qquad + \int_0^t \|M'(s) (J^{-}(G,n)(T) - J^{-}(G,n)(t))\| \,ds
\\ & \leq C K \Big[ \int_t^T (T-s)^{-\delta} (T-s)^{\alpha-\frac1p}\, ds + (T-t)^{\alpha-\frac1p} \int_0^t (T-s)^{-\delta} \,ds\, \Big]
\\ & \leq C K \Big[ (T-t)^{\alpha-\frac1p-\delta+1} + [T^{-\delta+1} + (T-t)^{-\delta+1}]  (T-t)^{\alpha-\frac1p}\Big],
\end{align*}
and the latter goes to zero as $t\uparrow T$.

We conclude that almost surely for every $n\geq 1$
\begin{equation}\label{eq:intbypartforn}
I^{-}(M G,n) =  M(0)I^{-}(G,n) + \int_0^T M'(s) I^{-}(\one_{[s,T]}G,n) \,ds.
\end{equation}
Hence to prove \eqref{eq:forwardintbyparts}, we will show that we can let $n\to \infty$ in \eqref{eq:intbypartforn}. Obviously, $M(0)I^{-}(G,n)\to M(0)I(G)$. From Theorem \ref{thrm:frwrd-ito-coincide-p>2} and \eqref{eq:fractionalSobolevemb}
we find that $\xi_n = [J^{-}(G,n) - J(G)]_{C^{\alpha-\frac1p}(0,T;X)}\to 0$ in probability as $n\to \infty$. It follows that
\begin{align*}
&\int_0^T \big\|M'(s) [I^{-}(\one_{[s,T]}G,n) - I^{-}(\one_{[s,T]}G)]\big\|\,ds
\\ & \leq C \int_0^T (T-s)^{-\delta} \|I^{-}(\one_{[s,T]}G,n) - I^{-}(\one_{[s,T]}G)\|\,ds
\\ & \leq C \int_0^T (T-s)^{-\delta} \|(J^{-}(G,n)(T) - J^{-}(G,n)(s)) - (J^{-}(G)(T) - J^{-}(G)(s))\|\, ds
\\ & \leq C \xi_n \int_0^T (T-s)^{-\delta+\alpha-\frac1p} \,ds
\\ & = C' \xi_n T^{1-\delta+\alpha-\frac1p}.
\end{align*}
Since the latter converges to zero in probability, it follows that the right-hand side of \eqref{eq:intbypartforn} converges and hence $MG$ is forward integrable and
\eqref{eq:forwardintbyparts} holds.
\end{proof}

\begin{remark}\label{rem:suffweaklL2}
Assume $M$ satisfies (i) and (ii) of Theorem \ref{thm:productforward}.
\begin{enumerate}[(1)]
\item If $\delta\in [0,1)$, then by Lemma \ref{lem:intbyparts} one has $MG\in L^0(\OO;\g(0,T;H,Y))$ whenever $G\in L^0(\OO;\g(0,T;H,Y))$. In particular $MG$ is weakly in $L^2(0,T;H)$.
\item If $0\leq \delta <\tfrac32 - \tfrac1p$ and $G\in L^0(\OO;L^p(0,T;\g(H,X)))$, then we have $MG\in L^0(\OO;L^2(0,T;\g(H,Y)))$. Indeed, without loss of generality we can assume $\delta> 1$. It follows that
\begin{align*}
\|M(t)-M(0)\| \leq C\int_0^t (T-s)^{-\delta} \, ds \leq C \big((T-t)^{1-\delta} + T^{1-\delta}\big).
\end{align*}
Therefore, by H\"older's inequality with $\frac{1}{q} + \frac{2}{p} = 1$,
\begin{align*}
\|MG\|_{L^2(0,T;\g(H,Y))} & \leq C \Big(\int_0^T \big((T-t)^{1-\delta} + T^{1-\delta}\big)^2 \|G(t)\|^2_{\g(H,X)} \, dt\Big)^{1/2} \\&\qquad+ \|M(0)\|\Big(\int_0^T \|G(t)\|^2_{\g(H,X)} \, dt \Big)^{1/2}
\\ & \leq C \|G\|_{L^p(0,T;\g(H,X))}.
\end{align*}
\end{enumerate}
\end{remark}

From Theorem \ref{thm:productforward}, Proposition \ref{prop:embeddingtype} and Remark \ref{rem:suffweaklL2} we immediately derive the following:
\begin{corollary}\label{cor:productforward}
Assume $X$ and $Y$ are {\umd} Banach space with type $2$ and assume $M$ satisfies (i) and (ii) of Theorem \ref{thm:productforward}.
Assume $p>2$ and $\delta<\tfrac32 - \tfrac1p$.
If $G\in L^0(\OO;L^p(0,T;\g(H,X)))$ is adapted, then $M G$ is forward integrable, $s\mapsto M'(s) I(\one_{[s,t]}G)\in L^1(0,T;Y)$ almost surely, and \eqref{eq:forwardintbyparts} holds.
\end{corollary}

As an illustration we present a brief indication how the results of this section can be applied to stochastic evolution equations.
\begin{example}
Assume that for each $\omega\in \OO$, $(A(t,\omega))_{t\in [0,T]}$ is a family of unbounded operators which generates an evolution family $(S(t,s,\omega))_{0\leq s\leq t\leq T, \omega\in\OO}$ on a Banach space $X_0$. Assume that $X_1 = D(A(t,\omega))$ does not depend on time and $\omega\in \OO$, and $A:[0,T]\times\OO\to\calL(X_1,X_0)$ is adapted.
In general, $\omega \mapsto S(t,s,\omega)$ will only be $\F_t$-measurable, and hence the stochastic convolution
\[\int_0^t S(t,s) G(s) \, d W(s)\] does not exist as an It\^o integral. In many situations one can check that $\frac{d}{ds}S(t,s) = - S(t,s) A(s)$ satisfies $\big\|\frac{d}{ds}S(t,s,\omega)\big\|\leq C(\omega) (t-s)^{-1}$ (see \cite{AT1} and \cite{Lunardi}). Therefore, Theorems \ref{thm:productforward} and Corollary \ref{cor:productforward} with $M(s) = S(t,s)$ can be used to obtain sufficient conditions for the existence of the forward convolution
\begin{equation}\label{eq:reprformFwdInt}
\begin{aligned}
U(t) & := \int_0^t S(t,s) G(s) \, d^{-}W(s)
\\ & = S(t,0)I(\one_{[0,t]}G) - \int_0^t S(t,s) A(s) I(\one_{[s,t]}G)  \;ds.
\end{aligned}
\end{equation}
In \cite{NualartLeon} Le{\'o}n and Nualart have observed that the forward integral gives a weak solution of the stochastic evolution equation
\[d U = A(t) U(t) \, dt + G(t)\, d W(t),  \ \ U(0)= 0,\]
and even more general equations.
Using \eqref{eq:reprformFwdInt} one can obtain a rather complete theory for non-autonomous stochastic evolution equations with random drift.
Details can be found in \cite{PV-representation}.
\end{example}

\def\polhk#1{\setbox0=\hbox{#1}{\ooalign{\hidewidth
  \lower1.5ex\hbox{`}\hidewidth\crcr\unhbox0}}}

\end{document}